\documentclass{amsproc}
\usepackage{euscript}
\usepackage{cases}
\usepackage{mathrsfs}
\usepackage{bbm}
\usepackage{amssymb}
\usepackage{amsfonts,amsmath,amsxtra,mathdots,mathabx}
\usepackage{color}
\usepackage{hyperref}
\usepackage{tikz}
\usepackage{appendix,upgreek}

\allowdisplaybreaks

\DeclareFontFamily{U}{matha}{\hyphenchar\font45}
\DeclareFontShape{U}{matha}{m}{n}{
	<5> <6> <7> <8> <9> <10> gen * matha
	<10.95> matha10 <12> <14.4> <17.28> <20.74> <24.88> matha12
}{}
\DeclareSymbolFont{matha}{U}{matha}{m}{n}

\DeclareMathSymbol{\Lt}{3}{matha}{"CE}
\DeclareMathSymbol{\Gt}{3}{matha}{"CF}

\DeclareSymbolFont{mathc}{OML}{txmi}{m}{it}
\DeclareMathSymbol{\varvv}{\mathord}{mathc}{118}
\DeclareMathSymbol{\varww}{\mathord}{mathc}{119}
\DeclareMathSymbol{\vnu}{\mathord}{mathc}{"17}

\def\varnu{\text{\scalebox{1.05}{$\vnu$}} }

\DeclareSymbolFont{mathd}{OML}{ztmcm}{m}{it}
\DeclareMathSymbol{\varalpha}{\mathord}{mathd}{11}
\DeclareMathSymbol{\varlambda}{\mathord}{mathd}{21}

\DeclareMathSymbol{\depsilon}{\mathord}{mathd}{15}

\def\vepsilon{\upvarepsilon} 

\DeclareMathSymbol{\varchi}{\mathord}{mathd}{31}
 
\def\valpha{\text{\scalebox{0.88}[1.02]{$\alpha$}}} 

\newcommand{\BA}{{\mathbb {A}}} 
\newcommand{\BC}{{\mathbb {C}}}

\newcommand{\BQ}{{\mathbb {Q}}} \newcommand{\BR}{{\mathbb {R}}}

 \newcommand{\BZ}{{\mathbb {Z}}}

 \newcommand{\RN}{{\mathrm {N}}}

\newcommand{\GL}{{\mathrm {GL}}}
\newcommand{\PGL}{{\mathrm {PGL}}}
\newcommand{\SL}{{\mathrm {SL}}}

\newcommand{\sstyle}{\scriptstyle}

\newcommand{\ra}{\rightarrow} 

\def\nnmid{\hskip -1.5 pt \nmid \hskip -1.5 pt}

\def\fra{\mathfrak{a}}

\def\frn{\mathfrak{n}}

\def\-{^{-1}}

\def\nd{\mathrm{d}}
\def\Tr{\mathrm{Tr}}

\def\Fx{F^{\times}}

\def\sstimes {\scalebox{0.55}{$\times$}}

\renewcommand{\Re}{{\mathrm{Re} }}

\def\shskip{\hskip 0.5 pt}

\def\frO {\text{\raisebox{- 2 \depth}{\scalebox{1.1}{$ \text{\usefont{U}{BOONDOX-calo}{m}{n}O} \hskip 0.5pt $}}}}

\def\frp{\mathfrak{p}}

\def\frb{\mathfrak{b}}

\def\frd{\mathfrak{d}}

\def\frD{\mathfrak{D}}

\def\ree{\mathrm{e}}

\def\SO{\text{\raisebox{- 2 \depth}{\scalebox{1.1}{$ \text{\usefont{U}{BOONDOX-calo}{m}{n}O} \hskip 0.5pt $}}}}

\def\SDB{\text{\raisebox{- 2 \depth}{\scalebox{1.1}{$ \text{\usefont{U}{dutchcal}{m}{n}B}  $}}}}

\newcommand{\delete}[1]{}

\theoremstyle{plain}

\newtheorem{thm}{Theorem}[section] \newtheorem{cor}[thm]{Corollary}
\newtheorem{lem}[thm]{Lemma}  \newtheorem{prop}[thm]{Proposition}

 \newtheorem{defn}[thm]{Definition}

\newtheorem {rem}[thm]{Remark}

\newtheorem*{acknowledgement}{Acknowledgements}

\numberwithin{equation}{section}

\begin{document}
	
\title[A Vorono\"i--Oppenheim Summation Formula]{{\large A Vorono\"i--Oppenheim Summation Formula for Number Fields}}

\author{Zhi Qi}
\address{School of Mathematical Sciences\\ Zhejiang University\\Hangzhou, 310027\\China}
\email{zhi.qi@zju.edu.cn}

\thanks{The  author was supported by the National Natural Science Foundation of China (Grant No. 12071420).}

\keywords{Vorono\"i summation, Bessel functions, Eisenstein series.}
\subjclass[2010]{11F70}

\begin{abstract}
	In this note, we establish a Vorono\"i--Oppenheim summation formula for divisor functions over an arbitrary number field. 
\end{abstract} 

	\maketitle



\section{Introduction}

In 1904, Vorono\"i \cite{Voronoi} introduced his famous summation formula for the classical divisor function $\tau (n) $, whose smoothed form (see \cite[(1.5, 1.6)]{Templier-GL2} and \cite[\S 4.5]{IK}) reads as follows:
\begin{align}\label{app: Voronoi}
	\begin{aligned}
		\sum_{n=1}^{\infty} \tau (n) \varww (n) = & \int_0^{\infty} \varww (x) (\log x + 2 \gamma)  \nd x \\
		& + \sum_{n=1}^{\infty} \tau (n) \int_0^{\infty} \varww (x) \big(4 K_0 (4\pi \sqrt{n x }) -2 \pi Y_0  (4\pi \sqrt{n x }) \big)  \nd x, 
	\end{aligned}
\end{align}
for  $\varww (x) \in C_c^{\infty} (0, \infty)$, in which $\gamma$ is Euler's constant.

In 1927, Oppenheim \cite{Oppenheim} extended Vorono\"i's summation formula for  $$\tau_s (n) =  \sum_{ab=n} (a/b)^{s} = n^{-s} \sum_{d|n} d^{2s}, \qquad \text{($s\in \BC$)}, $$ as follows: 
\begin{align}\label{app: Oppenheim}
	\begin{aligned}
		\sum_{n=1}^{\infty} \tau_s (n) & \varww (n) =           \int_0^{\infty} \varww (x) \big({\zeta (1-2 s)} x^{- s} + {\zeta (1+2s)}    x^{s} \big)   \nd x  \\
		&\qquad \quad + \sum_{n=1}^{\infty} \tau_s (n) \int_0^{\infty} \varww (x) \big\{ {4 \cos (\pi s)}    K_{2 s} (4 \pi \sqrt {nx })   \\
		& \qquad \quad - 2\pi \big( \cos (\pi s) Y_{2s} (4 \pi \sqrt {nx }) + \sin  (\pi s) J_{2s} (4 \pi \sqrt {nx }) \big)  \big\}  \nd x.
	\end{aligned} 
\end{align}

In this note, we generalize the Vorono\"i--Oppenheim formula to an arbitrary number field. Actually,  our formula is even more general, with additive twists  included (see  \cite[\S 4.5]{IK})---this feature is usually important for applications\footnote{Our motivation of writing this note was the application of Vorono\"i over an imaginary quadratic field in \cite{Qi-Liu-Moments}, where it was proven that at least $33\%$ of central $L$-values for $\PGL_2 (\frO)$-Maass forms are non-vanishing (here $\frO$ is the ring of integers in an imaginary quadratic field).}. Our proof is inspired by the adelic  approach to the Vorono\"i summation formula for cusp forms   in  Cogdell \cite{Cogdell-Bessel} and Templier \cite{Templier-GL2}. In our setting, Eisenstein series are used instead of cusp forms.  For the archimedean vectors, we use the constructions in Beineke--Bump \cite{BB-VO1} and extend their  result on the Whittaker integral to complex places by a kernel formula for $\GL_2 (\BC)$ established in \cite{Qi-Bessel}. Recently,  the ideas in \cite{BB-VO1} were used in  \cite{BBT-VO} and \cite{BBB-VO} to establish a Vorono\"i--Oppenheim formula over   totally real number fields and the Gaussian field, but our adelic approach is  conceptually simpler while our formula is more general.

\subsection*{Notation and Definitions}

Let $F    $ be a number field. Let $\SO $,  $\mathfrak{D}$, and $ \BA  $ be its ring of integers,   different ideal, and     adele ring. Let $\mathrm{N}$ denote the norm for $F    $.

 For each place $v$ of $F    $, we denote by $F    _{v}$ the corresponding local field. When $v$ is non-archimedean,  let $\frp_{v}$ be the corresponding prime ideal of   $\SO $ and let $\mathrm{ord}_{v}$   denote the additive valuation.  Let $\| \hskip 3  pt \|_{v}$  denote the normalized modulus of $F_{v}$. We have $\| \hskip 3  pt \|_{v} = | \hskip 3.5 pt |$ if $F    _{v} = \BR$ and  $\| \hskip 3  pt \|_{v} = | \hskip 3.5 pt |^2$ if $F    _{v} = \BC$, where $| \hskip 3.5  pt |$ is the usual absolute value.

Let $S_{\infty}$ or $S_f$ denote the set of archimedean or non-archimedean  places of $F    $, respectively. Write $v | \infty$ and $v     \nnmid    \infty$ as the abbreviation for $v \in S_{\infty}$ and $v \in S_{f}$, respectively. For a finite set of places $S$,  
 denote by $\BA^S$, respectively $F    _{S}$, the sub-ring of adeles with trivial component above $S$, respectively above the complement of $S$. 
 For brevity, write $\BA_f = \BA^{S_{\infty}}$ and $F    _{\infty} = F    _{S_{\infty}}$. The modulus on $F    _{\infty}$  will be denoted by $\| \hskip 3  pt \|_{\infty}$. 

Let $\ree (z) = \exp (2\pi i z)$. Fix the (non-trivial) standard additive character $\psi = \otimes_{v} \psi_{v}$ on $\BA/F    $ as in \cite[\S XIV.1]{Lang-ANT} such that $\psi_{v} (x) = \ree (- x)$ if $F    _v = \BR$,   $\psi_{v} (z) = \ree (- (z + \widebar z))$ if $F    _v = \BC$, 
and that $\psi_{v}$ has conductor $\mathfrak{D}_{v}\-$ for any non-archimedean $F    _v$. We split $\psi = \psi_{\infty} \psi_f  $ so that $\psi_{\infty} (x) = \ree (- \Tr_{F_{\infty}} (x))$ ($x \in F_{\infty}$). For a finite set of places $S$, define  $\psi_S  = \prod_{v \shskip  \in S} \psi_{  v}  $ as an additive character of $F_{S}$.

We choose the Haar measure  $\nd x$ of $F    _{v}$ self-dual with respect to $\psi_{v}$ as in  \cite[\S XIV.1]{Lang-ANT}; the Haar measure is the ordinary Lebesgue measure on the real line if $F    _{v} = \BR$, and twice  the ordinary Lebesgue measure on the complex plane if $F    _{v} = \BC$. The measure $\nd x$ on $F_{\infty}$ is defined to be the product   of $ \nd x_{v} $ for $v |\infty$. 

In general, we use Gothic letters $\fra ,  \frb , \dots$ to denote {non-zero} fractional ideals of $F$, while we reserve  $\frn$ and $\frd$  for non-zero integral ideals of $F$. Let $\RN (\fra)$ denote the norm of $\fra$.

Let $\zeta_F (s)$ be the Dedekind $\zeta$ function for $F$:
\begin{align*}
	\zeta_F (s) = \sum_{\frn \shskip \subset \shskip \frO }  \frac { 1 } {\RN(\frn)^{  s} }, \qquad \mathrm{Re} (s) > 1. 
\end{align*}
It is well-known that  $\zeta_F (s)$ is a meromorphic function on the complex plane with a simple pole at $s=1$. Let  $\gamma^{(-1)}_F $ and $\gamma^{(0)}_F $ respectively be the residue and the constant term  of  $\zeta_F (s)$  at  $s=1$; namely, 
\begin{align}\label{2eq: zeta (s), s=1}
	\zeta_F (s) = \frac {\gamma^{(-1)}_F} {s-1} + \gamma^{(0)}_F + O (|s-1|), \qquad s \ra 1. 
\end{align}

\begin{defn}
	[Bessel kernel] \label{defn: Bessel kernel}
	
	Let  $s  \in   \BC $. 
	
	{\rm(1)} When  $F_v = \BR$, for $x  \in \BR_+$ we define 
	\begin{align*}
		 & B_{s} (x)   = \frac {\pi} {\sin (\pi s) } \big( J_{-2 s} (4 \pi \sqrt {x }) - J_{2 s} (4 \pi \sqrt {x }) \big) ,  \\
		 & B_{s} (-x )      = \frac {\pi} {\sin (\pi s) } \big( I_{-2 s} (4 \pi \sqrt {x }) - I_{2 s} (4 \pi \sqrt {x }) \big)  .
	\end{align*}
	
	{\rm(2)} When  $F_v = \BC$, for $z \in \BC^{\times}$ we define
	\begin{equation*}
		B_{s} (z ) =  \frac {2\pi^2} {\sin (2\pi s) } \big( { \textstyle  J_{-2 s} (4 \pi \sqrt {z}) J_{- 2s} (4 \pi \sqrt { \widebar z}) - J_{2 s} (4 \pi \sqrt {z}) J_{ 2s} (4 \pi \sqrt { \widebar z}) } \big). 
	\end{equation*}  
	
	
	For $x \in F_{\infty}^{\times}$  we define
	\begin{align*}
		\SDB_{s} (x ) = \prod_{ v | \infty } B_{s} (x_{v} ) .
	\end{align*}
\end{defn}	

It is understood that when $   s \in \BZ $ or $  2 s \in \BZ $ in {\rm(1)} or {\rm(2)} in Definition \ref{defn: Bessel kernel}, respectively,  the formulae above should be replaced by their limit. Alternatively, by \cite[3.54 (1), 3.7 (6)]{Watson}, we obtain the the expressions that arise in \eqref{app: Oppenheim}: 
\begin{align*}
& B_s (x)	  = - 2\pi \big( \cos (\pi s) Y_{2s} (4 \pi \sqrt {x }) + \sin  (\pi s) J_{2s} (4 \pi \sqrt {x }) \big) , \\
& B_s (-x)   = {4 \cos (\pi s)}    K_{2 s} (4 \pi \sqrt {x }) . 
\end{align*}

\begin{defn}[Hankel transform and Mellin transform]\label{defn: Hankel transform}
	Let	$ \mathscr{C}^{\infty}_c (F^{\times}_{\infty}) $ denote the space of compactly supported smooth functions $\varww : F^{\times}_{\infty} \ra \BC $ that are of the  product form $
	\varww (x)   = \allowbreak \prod_{v | \infty} \varww_{v} (x_{v})$. 
	
	Let $s  \in \BC $.	For $\varww (x) \in \mathscr{C}^{\infty}_c (F^{\times}_{\infty}) $  we define its Hankel transform $ \widetilde{\varww}_{s} (y) $ and Mellin transform $ \widetilde{\varww}_{s} (0) $  by
	\begin{align*}
		\widetilde{\varww}_{s} (y) =    \int_{F^{\sstimes}_{\scalebox{0.55}{$\infty$} } }  \varww (x) \SDB_{s}    ( x y)   \nd x, \quad \widetilde{\varww}_{s} (0) = \int_{F^{\sstimes}_{\scalebox{0.55}{$\infty$} } }  \varww (x) \|x\|_{\infty}^{s}    \nd x, \qquad  y \in F^{\times}_{\infty} .
	\end{align*} 
\end{defn}

\subsection*{Statement of Results}

Our main result is the following summation formula.\footnote{Edgar Assing informed the author that a general Vorono\"i--Oppenheim summation formula for twisted divisor sums can be obtained in a similar fashion by considering certain ramified Eisenstein series and using the local computations in his work \cite{Assing-Eisenstein}.}

\begin{thm}\label{prop: V-O} 
 	Let $ \zeta \in F$. For a non-zero fractional ideal $\fra $ define \begin{align}\label{1eq: S and b}
 	 S = \big\{ v \nnmid \infty : \mathrm{ord}_{v} (\zeta) < \mathrm{ord}_{v} (\fra) \big\}, \qquad \frb = \fra^{-1}  \prod_{v \shskip \in S}  \frp_{v}^{   2 \shskip \mathrm{ord}_{v} ( (1/\zeta) \fra)      } .   
 	\end{align}	 
 For $s \in \BC$  define  
	\begin{align}
		\tau_s (\frn ) = \RN (\frn)^{-s} \sum_{   \frd | \frn  }  \RN (\frd)^{ 2 s }, 
	\end{align}
and let $ \varww   (x) $, $\widetilde{\varww}_{s} (0)$, and $\widetilde{\varww}_{s} (y)$ be as in Definition {\rm\ref{defn: Hankel transform}}.
	  Then we have the identity
	\begin{align}\label{app: Voronoi, tau s} 
		\begin{aligned}
			&  	 \sum_{\gamma \shskip \in (\fra\frD)^{-1} \smallsetminus \{0\} }   \frac { \psi_{\infty} (  \gamma \zeta )      \tau_s (\gamma \fra \frD)  \varww   (\gamma )  } {\sqrt{\RN (\fra  ) }}     \\
		=	&   \sum_{\pm}   \frac{\RN (  \frD)^{\frac 1 2 \pm s}}{\RN ( \frb )^{\frac 1 2 \pm s}}   \zeta_F (1\pm 2s)  \widetilde{\varww}_{\pm s}  (0)    	+    \hskip -2 pt \sum_{\gamma \shskip \in  (\frb \frD)^{-1} \smallsetminus \{0\} }  \hskip -2 pt  \frac {\psi_S  (     \gamma / \zeta ) \tau_{s} (\gamma   \frb \frD ) \widetilde{\varww}_{s} (\gamma)} {{ \sqrt{\RN (\frb  ) } }} .
		\end{aligned}
	\end{align} 
\end{thm}

By letting $\zeta = 0$ and $\fra = (1)$ in   \eqref{app: Voronoi, tau s}  (it is understood that if  $\zeta = 0$ then $S = \text{{\rm\O}}$, $\frb = (1)$, and $\psi_{\text
	{\O}} = 1$), we recover   the Vorono\"i--Oppenheim formula \eqref{app: Oppenheim} when $F = \BQ$ as well as its generalization in \cite{BBT-VO} when $F$ is totally real.  

Let $\tau (\frn) = \tau_0 (\frn)$ be the (usual) divisor function for $F$.   The following Vorono\"i summation formula is  the  formula \eqref{app: Voronoi, tau s}   in the special case $s = 0$ (see   \eqref{10eq: zeroth term s=0}). When $F = \BQ$, this is  the Vorono\"i summation formula in \cite[\S 4.5]{IK}. 

\begin{cor} \label{prop: Voronoi tau}	Let $\zeta$, $\fra$, $\frb$, $S$ be as in Theorem {\rm\ref{prop: V-O}}.	Let $ \varww   (x) $ and $\widetilde{\varww}_{0} (y)$ be as in Definition {\rm\ref{defn: Hankel transform}}. Define 
	\begin{align}\label{3eq: zeroth term}
		\widetilde{\varww}_{0} (0; \frb ) =	   \gamma^{(-1)}_F  \widetilde{\varww}_0' (0) + \big( 2 \gamma^{(0)}_F  -\gamma^{(-1)}_F  \log \RN \big(\frb\frD^{-1}  \big) \big) \widetilde{\varww}_0 (0),
	\end{align}
	where the constants $\gamma^{(-1)}_F $ and $\gamma^{(0)}_F $ are defined as in {\rm\eqref{2eq: zeta (s), s=1}}, and $ \widetilde{\varww}_0 (0)$ and  $ \widetilde{\varww}_0' (0)$ are the integrals
	\begin{align}\label{3eq: Mellin}
		\widetilde{\varww}_0 (0) =	 \int_{F^{\sstimes}_{\scalebox{0.55}{$\infty$} } }  \varww (x)   \nd x , \qquad \widetilde{\varww}_0' (0) =	 \int_{F^{\sstimes}_{\scalebox{0.55}{$\infty$} } }  \varww (x)    \log   \|x\|_{\infty}  \nd x . 
	\end{align}   Then we have the identity
	\begin{align}\label{app: Voronoi, tau} 
		\begin{aligned}
			 { \frac { 1  } { \sqrt{\RN (\fra  ) } } }	 \sum_{\gamma \shskip \in (\fra\frD)^{-1} \smallsetminus \{0\} }     \psi_{\infty} (  \gamma \zeta )      \tau  (\gamma \fra \frD) & \varww   (\gamma )        
			=	   \frac {\sqrt{\RN(\frD)}} {\sqrt{\RN (\frb )}} \widetilde{\varww}_{0} (0; \frb )   \\
			+  \frac 1 {\sqrt{\RN (\frb )}}  & \sum_{\gamma \shskip \in  (\frb \frD)^{-1} \smallsetminus \{0\} }  \hskip -2 pt  \psi_S  (     \gamma / \zeta ) \tau (\gamma   \frb \frD ) \widetilde{\varww}_{0} (\gamma)  .
		\end{aligned}
	\end{align}
	
\end{cor}

\begin{rem}
	With some efforts, one may prove that the above formulae are valid for any $\varww : F^{\times}_{\infty} \ra \BC $ with compact support. Note that for such $\varww$ the integral transforms in Definition {\rm\ref{defn: Hankel transform}} are still well-defined. 
\end{rem}

\begin{acknowledgement}
	The author  thanks  Edgar Assing and the referee for their helpful	comments.
\end{acknowledgement}

\section{Review of Eisenstein Series} 

In this section, we recollect some basic facts on Eisenstein series. The reader is referred to  \cite[\S 3.7]{Bump} for more details.

Define the parabolic subgroup
\begin{align*}  P = \left\{  \begin{pmatrix}
		x & r \\ & y
	\end{pmatrix}\right\} \subset \GL_2.
\end{align*}

 For $s \in \BC$  let $ \pi (s) $ be the representation of $\GL_2 (\BA)$ obtained by (normalized) parabolic induction of the following character of $P (\BA)$:
\begin{align*}
	\begin{pmatrix}
		x & r \\ & y 
	\end{pmatrix} \ra \|x/y\|^{s+\frac 1 2} .
\end{align*} 
To be precise, the space $V (s)$ of this representation consists of all smooth functions $\phi$ on $\GL_2 (\BA)$ that satisfy
\begin{align*}
	\phi (\begin{pmatrix}
		x & r \\ & y 
	\end{pmatrix} g ) = \|x / y\|^{s +   \frac 1 2} \phi (g) ,
\end{align*}
on which the action of $\GL_2 (\BA)$ is by right translation. 

For $\Re  ( s) > \frac 1 2$ define the Eisenstein series $E (g; \phi)$ associated with $\phi \in V (s)$ by 
\begin{align*}
	E (g; \phi) = \sum_{ \gamma \shskip \in P(F) \backslash \GL_2 (F) } \phi  (\gamma g), \qquad g \in \GL_2 (\BA); 
\end{align*}
the series is absolutely convergent. 
The Fourier--Whittaker expansion of $ E (g; \phi)  $ is given by
\begin{align}\label{2eq: Fourier exp Eisenstein}
	E (g; \phi) = \phi (g) +  M \phi  (g) + \sum_{\gamma \shskip \in F^{\times}} W_\phi (a (\gamma) g) , 
\end{align}
where 
\begin{align}\label{app: intertwining, A}
	M  \phi  (g) = \int_{ \BA } \phi (w \shskip n(r) g )   \nd r , \qquad \phi \in V (s), 
\end{align}
is the intertwining integral, 
and
\begin{align}\label{app: Whittaker function, A}
	W_\phi (g) =  \int_{\BA} \phi (w \shskip n(r) g ) \overline{\psi ( r)} \nd r , \qquad \phi \in V (s), 
\end{align}
is the Whittaker integral, 
with 
\begin{align*}
	a (x) = \begin{pmatrix}
		x & \\ & 1
	\end{pmatrix}, \qquad n (r) = \begin{pmatrix}
		1 & r \\
		& 1
	\end{pmatrix}, \qquad w = \begin{pmatrix}
		& \hskip -2pt -1 \\ 
		1 & 
	\end{pmatrix}. 
\end{align*}
It is known that the (global) integrals in \eqref{app: intertwining, A} and \eqref{app: Whittaker function, A} are   converge for $\mathrm{Re} (s) > \frac 1 2$ and have analytic continuation onto the whole complex plane, except for a simple pole at $s = \frac 1 2$ that occurs in the case of $M \phi (g)$. 
Moreover, we have $ M:  V(s) \ra V(-s) $, and  $W_{\phi} (n (r) g) = \psi (r) W_{\phi} (g)$.

\section{An Adelic Identity} 

The starting point of our approach  is the following identity, which is the analogue of \cite[Theorem 3.1]{Templier-GL2} in the case of Eisenstein series.

\begin{lem}\label{lem: E = Ew}
	Let $\zeta    \in \BA_f  $ and $\valpha \in \BA_f^{\times}$.  For $\phi \in V (s)$ we have 
	\begin{align}\label{app: basic identity}
		\begin{aligned}
			\phi (a (\valpha))  +   M & \phi  (a (\valpha)) +   \sum_{\gamma \shskip \in F^{\times}} \psi_f (\gamma \zeta   ) W_\phi (a (\valpha \gamma) ) = \\
			 &   \phi  (w n (\zeta   ) a (\valpha) ) +  M \phi  (w n (\zeta   ) a (\valpha)) + \sum_{\gamma \shskip \in F^{\times}}   W_\phi (a (\gamma) w n (\zeta   ) a (\valpha) ) .
		\end{aligned}
	\end{align}
\end{lem}

\begin{proof}
	In view of \eqref{2eq: Fourier exp Eisenstein}, it is easily seen that the left-hand side is   $E (n (\zeta   ) a (\valpha) ; \phi)$ while the right-hand side is $E (w n (\zeta   ) a (\valpha); \phi)$. Since $E (g; \phi)$ is left $\GL_2 (F)$-invariant, the identity   follows because $w \in \GL_2 (F)$. 
\end{proof} 

We now assume that $\phi \in V (s)$ is factorizable as $\prod_{v} \phi_{v}$; it is clear that $M \phi$ and $W_{\phi}$ are also factorizable. More precisely, for each $v$,  $\phi_{v}$ is a smooth function on $\GL_2 (F_{v})$ such that 
\begin{align*}
	\phi_{v} (\begin{pmatrix}
		x & r \\ & y 
	\end{pmatrix} g ) = \|x / y\|_{v}^{s +   \frac 1 2} \phi_{v} (g), \qquad x, y \in F_{v}^{\times}, \, r \in F_{v}, \, g \in \GL_2 (F_{v}) ;
\end{align*} by convention, let $V_{v}(s)$ denote the space of such $\phi_{v}$ with the above property.   Put $\phi_{\infty} = \prod_{v | \infty} \phi_{v}$, $M \phi_{\infty} = \prod_{v | \infty} M \phi_{v}$, and $ W_{\phi_\infty} = \prod_{v | \infty} W_{\phi_{v}}  $.

Choose $\valpha \in \BA_f^{\times} $ so that $\mathrm{ord}_v (\valpha_v) = \mathrm{ord}_v (\fra)$ for every $v \nmid \infty$. Define \begin{equation}\label{3eq: defn of b}
	S = \big\{ v \nnmid \infty : \mathrm{ord}_{v} (\zeta) < \mathrm{ord}_{v} (\valpha) \big\}, \qquad  \frb =   \prod  \frp_{v}^{  \max  \left \{   \mathrm{ord}_{v} (1/\valpha ), \, \mathrm{ord}_{v} (\valpha/\zeta^2)   \right \}  }.  
\end{equation}
For every $v \nnmid \infty$,   choose $\phi_{v} = \phi_{s, \shskip   v}$ to be the canonical spherical vector in $V_{v} (s)$ with $ \phi_{s, \shskip  v} (k) =1 $ for all $k \in \GL_2(\SO_{v})$; namely,
\begin{align*}
	\phi_{s, v} (\begin{pmatrix}
		x & r \\ & y 
	\end{pmatrix} k ) = \|x / y\|_{v}^{s +  \frac 1   2} , \qquad x, y \in F_{v}^{\times}, \, r \in F_{v}, \, k \in \GL_2 (\SO_{v}) . 
\end{align*}  The local integrals $ M \phi_{v} $  and  $ W_{\phi_{v}}$ are very explicit in the spherical case (see for example  \cite[\S 4.6]{Bump}\footnote{A subtle issue is that the results in \cite[\S 4.6]{Bump} are proven for $\psi_{ v }$ of conductor $\frO_{v}$, but this may be easily addressed by re-scaling the character and the Haar measure.}).   Globally, if we define 
\begin{align}
	c_{s} (0) = \frac {  \zeta_{F} (2s)  } {\sqrt{\RN (  \frD )} \zeta_{F} (1+2s ) }, \quad c_s (\frn) = \frac {   \tau_s (\frn) / \sqrt{\RN (\frn)}  } {  \RN (\frD)^{s }  \zeta_F(1+2s) } ,
\end{align} 
then
\begin{align*}
  \phi_f (a (\valpha)) = \frac {1} {\RN(\fra)^{   \frac 1 2 +s }} , 
\end{align*}
by definition, and 
\begin{align*}
  M  \phi_f  (a (\valpha)) =  \frac {   c_{s} (0) } { \RN(\fra)^{   \frac 1 2 -s  }  }, \qquad W_{\phi_{f}} (a ( \valpha \gamma) )  = c_s (\gamma \fra \frD), 
\end{align*}
by Proposition 4.6.7 and Theorem 4.6.5 in \cite{Bump}.  Consequently,  the left-hand side of \eqref{app: basic identity} is equal to 
\begin{align}\label{app: LHS}
	\frac {\phi_{\infty}  (1_2)} {\RN(\fra)^{   \frac 1 2 +s }}   +  c_{s} (0) \frac {  M \phi_{\infty}  (1_2)} { \RN(\fra)^{   \frac 1 2 -s  }  }    + \sum_{\gamma \shskip \in (\fra\frD)^{-1} \smallsetminus \{0\} }    { \psi_f ( \gamma \zeta ) c_s (\gamma \fra \frD) W_{\phi_{\infty}} (a (  \gamma) ) }   .
\end{align}
Next, we compute the right-hand side of  \eqref{app: basic identity}.  Keep in mind that $\phi_{v}$,  $M \phi_{v}$, and   $W_{\phi_{v}}$ are right $\GL_2 (\SO_{v})$-invariant. Recall that $S$ is defined in \eqref{3eq: defn of b}. When $v \in S_f \smallsetminus S$ so that $\|\zeta    / \valpha\|_{v} \leqslant 1 $, we have the Iwasawa decomposition 
\begin{align*}
	a(\gamma) w n   (\zeta   _{v}) a(\valpha_{v}) = \begin{pmatrix}
		\gamma  &  \\  &   \valpha_{v}
	\end{pmatrix}   
	\begin{pmatrix}
		&  -1 \\  1 &  \zeta   _{v} / \valpha_{v}
	\end{pmatrix} ,
\end{align*}  
and hence
\begin{align*}
	& \phi_{v}  (w n (\zeta   _{v}) a (\valpha_{v}) ) = \|1/\valpha\|_{v}^{ \frac 1 2+s}, \quad M \phi_{v}  (w n (\zeta   _{v}) a (\valpha_{v}) ) = \|1/\valpha\|_{v}^{\frac 1 2-s} M \phi_{v} (1_2),  
\end{align*}
\begin{align*}
	W_{\phi_{v}} (a (\gamma) w n (\zeta   _{v}) a (\valpha_{v}) ) = W_{\phi_{v}} (a (\gamma /  \valpha_{v}) ) . 
\end{align*}
When $v \in S$ so that $\|\zeta    / \valpha\|_{v} > 1 $, we have  the Iwasawa decomposition  
\begin{align*}
	a (\gamma ) w n(\zeta   _{v}) a(\valpha_{v})  =  \begin{pmatrix}
		1 & - \gamma / \zeta   _{v}  \\ & 1
	\end{pmatrix} 
	\begin{pmatrix}
		\gamma \valpha_{v} / \zeta   _{v}  & \\ &   \zeta   _{v} 
	\end{pmatrix} 
	\begin{pmatrix}
		1 & \\  \valpha_{v} / \zeta   _{v}  & 1
	\end{pmatrix}  ,
\end{align*}
and hence
\begin{align*}
	& \phi_{v}  (w n (\zeta   _{v}) a (\valpha_{v}) ) = \big\|\valpha / \zeta   ^2 \big\|_{v}^{\frac 1 2+s}, \quad M \phi_{v}  (w n (\zeta   _{v}) a (\valpha_{v}) ) = \big\|\valpha / \zeta   ^2 \big\|_{v}^{\frac 1 2-s} M \phi_{v} (1_2),  
\end{align*}
\begin{align*}
	W_{\phi_{v}} (a (\gamma) w n (\zeta   _{v}) a (\valpha_{v}) ) = \psi_{v} (-\gamma  /\zeta   _{v}) W_{\phi_{v}} \big( a \big(\gamma    \valpha_{v} / \zeta   _{v}^2 \big) \big) . 
\end{align*} 
In view of the definition of  $\frb$ in \eqref{3eq: defn of b}, it readily follows that
\begin{align*}
	\phi_f  (w n (\zeta   ) a (\valpha) ) = \frac 1 {\RN (\frb)^{\frac 1 2 + s}} ,
\end{align*}
and 
\begin{align*}
	M \phi_f  (w n (\zeta   ) a (\valpha)) = \frac {c(0)} {\RN (\frb)^{\frac 1 2 - s}}, \quad W_{\phi_f} (a (\gamma) w n (\zeta   ) a (\valpha) ) =  \psi_S  ( -   \gamma / \zeta ) c_{s} (\gamma   \frb \frD ). 
	\end{align*}
Therefore,   the    right-hand side of  \eqref{app: basic identity} is equal to
\begin{align}\label{app: RHS}
	\begin{aligned}
		\frac {\phi_{\infty} (w) } {\RN  (  \frb)^{\frac 1 2+s}  }  & + c_s (0)  \frac {  M \phi_{\infty} (w)} { \RN  (\frb )^{\frac 1 2- s} }   +  \sum_{\gamma \shskip \in  (\frb \frD)^{-1} \smallsetminus \{0\} }   {\psi_S  ( -   \gamma / \zeta ) c_{s} (\gamma   \frb \frD ) W_{\phi_{\infty}} (a (  \gamma) w ) }  .
	\end{aligned}
\end{align}  
Lemma \ref{lem: E = Ew} says that \eqref{app: LHS} and \eqref{app: RHS} are equal to each other.

For each  $v | \infty$,  we will make the choice of $\phi_{v}$ later in \S \ref{app: choice}.

\section{Archimedean Kirillov Model}  In this section, we will work exclusively on an archimedean local field $F_{v}   $. For simplicity, the place $v$ will be suppressed from our notation. Accordingly, let $F    $ be either $\BR$ or $\BC$. 
Let  $\psi (x) = \ree (- \Tr_{F   } (x))$, and $\nd x$ be the corresponding self-dual Haar measure on $F    $. Let $\| \phantom{i}  \| $  denote the standard modulus of $F    $. 
For   $s \in \BC$, let $V (s)$ be the space of {smooth} functions on $\GL_2 (F)$ that satisfy
\begin{align}\label{app: defn of V(s)}
	\phi (\begin{pmatrix}
		x & r \\ & y 
	\end{pmatrix} g ) = \|x / y\|^{s +   \frac 1 2} \phi (g),
\end{align} 
and let $\pi (s)$ denote the representation of  $\GL_2 (F)$ that acts on $V (s)$ by right translation. 

For simplicity, we assume that  $2 s \notin  \BZ \smallsetminus 2 \BZ$ or $2 s \notin \BZ \smallsetminus \{0\}$ according as $F$ is real or complex, so that $\pi (s)$ is irreducible.  

For $\Re (s) > 0$,   the Whittaker functional $L$ on $V(s)$ is defined by
\begin{align*}
	L (\phi) = \int_{F} \phi (w \shskip n(r)  ) \overline{\psi ( r)} \nd r, 
\end{align*}
in which  the integral  is convergent for $\Re (s) > 0$ (see \cite{Godement}). The Whittaker function $W_\phi  $ associated to $\phi \in V (s)$ is 
\begin{align}\label{app: Whittaker function}
	W_\phi (g) = L (\pi (g) \phi) = \int_{F} \phi (w \shskip n(r)  g ) \overline{\psi ( r)} \nd r . 
\end{align}
By definition, the Kirillov model $\mathscr{K}  (\pi (s))$ comprises all the functions
$   W_\phi (a(x))$   \text{($x \in F^{\times}$)}.     
It is known that $C_c^{\infty} (F^{\times}) \subset \mathscr{K}  (\pi (s))$ (see \cite[Lemma 5.1]{Ichino-Templier}). Moreover, we define the intertwining operator $ M  : V (s) \ra V (-s)$ by the integral
\begin{align}\label{app: intertwining}
	M   \phi  (g) = \int_{ F } \phi (w \shskip n(r) g )   \nd r .
\end{align}
Again, this integral is convergent for $\Re (s) > 0$. It is known that both the Whittaker integral
and the intertwining integral in \eqref{app: Whittaker function} and \eqref{app: intertwining}  have meromorphic continuation to the entire $s$
plane, but we will not need this fact since  $\Re (s) > \frac 1 2$ will be assumed.

\vskip 5 pt

\subsection{A Kernel Formula} We have the following kernel formula for the action of the Weyl element $w$ on the Kirillov 
model $ \mathscr{K}  (\pi (s)) $ as the  Hankel integral transform with   Bessel kernel $B_s$.  

\begin{prop}\label{lem: kernel formula} For $ W_\phi (a(x)) \in C_c^{\infty} (F^{\times})$, we have
	\begin{align}\label{app: pi(w)}
		W_\phi (a(y) w) 
		=   \int_{F    ^{\times}}  W_\phi (a(x))     B_{s} (xy) {\textstyle \sqrt{\|y / x\|}} \nd  x,  
	\end{align}
	where $B_s (x)$ is the Bessel kernel associated to $\pi (s)$ as in Definition {\rm\ref{defn: Bessel kernel}}. 
\end{prop}

\begin{proof}
	For the formula \eqref{app: pi(w)} in our case of $\pi (s)$, which is not unitary in general, we refer to Proposition 3.14, 3.17, Remark 17.6, and (18.1)--(18.4) in \cite{Qi-Bessel}.  
\end{proof}

The occurrence of Bessel functions in the representation theory of  $\SL_2 (\BR)$ may be traced back to the books of Gel$'$fand, Graev, and Piatetski-Shapiro \cite{GG-PS}, and Vilenkin \cite{Vilenkin-Special-Functions}.  We refer the reader to  \cite[\S \S 6, 8]{CPS}, \cite[Appendix 2]{BaruchMao-Real}, and \cite[\S \S 17, 18]{Qi-Bessel}  for the kernel formulae for unitary representations of $\GL_2 (\BR)$ and  $\GL_2 (\BC)$ (see also \cite{Mo-Kernel,Baruch-Hankel-DS,B-Mo-Kernel2,Baruch-Kernel} for $\SL_2 (\BR)$ and  $\SL_2 (\BC)$ (in special cases)).   For its applications in establishing the Kuznetsov formula and the Waldspurger formula, see  {\rm\cite{CPS,Qi-Kuz,BaruchMao-Real,Chai-Qi-Bessel,BaruchMao-Global,Chai-Qi-Wald}}.

For any infinite dimensional admissible representation of $\GL_2 (\BR)$ or $\GL_2 (\BC)$, it follows from  the Casselman--Wallach completion theorem {\rm(}see {\rm\cite{Casselman-CW,Wallach-CW}} and Chapter {\rm 11} in {\rm\cite{RRG-II})} that, after dividing $ \hskip -1 pt \sqrt{\|x\|} $,  the Kirillov model $\mathscr{K}$ is exactly the $\mathscr{S}_{\mathrm{sis}} $-space as defined in {\rm\cite{Qi-Bessel}} {\rm(}see also {\rm\cite{Miller-Schmid-2004})}. It should be stressed that for the unitary case,  the kernel formula is actually valid for all  $ W_\phi (a(x))   $ in  the Kirillov model. However, this is not necessarily true in general, and the action of $w$ needs to be interpreted in terms of $\GL_2 \times \GL_1$ local functional equations.

Finally, we conclude this sub-section with some discussions on the various proofs of this kernel formula in the literature. The case of $\GL_2 (\BR)$ or $\SL_2 (\BR)$ is relatively easier, and there are three proofs 
  in \cite[\S 8]{CPS},  \cite{Mo-Kernel}, and \cite[Appendix 2]{BaruchMao-Real}. The methods of the latter two proofs were generalized to  $\SL_2 (\BC)$ in  \cite{B-Mo-Kernel2}  and \cite{Baruch-Kernel}. However,  certain conditions are required due to some  convergence issues. In \cite{B-Mo-Kernel2}, an integral representation of the Bessel function is used but it is valid only for  $ |\Re (s) | < \frac 1 8$. In \cite{Baruch-Kernel}, it requires that $\Re (s)  \neq 0$ and hence the case of unitary principal series is excluded.\footnote{It should be noted that our parametrization is slightly different from theirs.} The approach in \cite{Qi-Bessel} is quite different and works without any condition. It is based on  the  sophisticated harmonic analysis for the Mellin transforms on $\mathscr{S}_{\mathrm{sis}} $-spaces (see \cite[\S 1--3]{Qi-Bessel}). Also the ideas in  \cite[\S 8]{CPS} are followed and generalized  in \cite[\S 17]{Qi-Bessel} to $\GL_n (\BR)$ and $\GL_n (\BC)$.

\vskip 5 pt

\subsection{Choice of Archimedean Vectors}\label{app: choice} Let $\varww \in C_c^{\infty} (F^{\times})$.  We define the function $\phi_{s, \varww}$ by 
\begin{align}\label{app: defn of phiw}
	\phi_{s, \varww} (g) = \left\{ \begin{aligned} \displaystyle &  \| x/y\|^{s+\frac 1 2} \int_{ F } \varww ( v  ) \|v \|^{s}   {\psi (  r  v)} \nd \shskip v
		, & & \text{ if } g = \begin{pmatrix}
			x &  u \\ & y
		\end{pmatrix} w \begin{pmatrix}
			1 & r \\ & 1
		\end{pmatrix}, \\
		& 0    , & & \text{ if } g = \begin{pmatrix}
			x & u \\ & y
		\end{pmatrix} .
	\end{aligned}\right. 
\end{align}

\begin{lem}
	We have $\phi_{s, \varww} \in V (s)$, that is,  $\phi_{s, \varww}$ is smooth and satisfies {\rm\eqref{app: defn of V(s)}}. 
\end{lem}

\begin{proof}
	It is clear that $\phi_{s, \varww}$  satisfies {\rm\eqref{app: defn of V(s)}}. It follows  that the smoothness of  $\phi_{s, \varww}$  is
	equivalent to the smoothness of its restriction to $\mathrm{SO}_2 (\BR)$ or $\mathrm{SU}_2 (\BC)$.  For $|a|^2 + |b|^2 = 1$ ($a, b \in \BR$ or $\BC$), with $b \neq 0$, if  we let  
	\begin{align*}
		\begin{pmatrix}
			a & - b \\
			\widebar{b} & \widebar{a}
		\end{pmatrix} = \begin{pmatrix}
			x &  u \\ & y
		\end{pmatrix} w \begin{pmatrix}
			1 & r \\ & 1
		\end{pmatrix},
	\end{align*}
	then $x = 1 / \widebar{b}$, $ y = \widebar{b} $, $u = a$, and $r =   \widebar{a } / \widebar{b}$. Therefore
	\begin{align*}
		\phi_{s, \varww} \begin{pmatrix}
			a & - b \\
			\widebar{b} & \widebar{a}
		\end{pmatrix} = \|b\|^{-2s-1} \int_{ F } \varww ( v  ) \|v \|^{s}   {\psi (    v \widebar{a } / \widebar{b} )} \nd \shskip v .
	\end{align*}
	The issue of smoothness is at  the points where $ b = 0  $, but the Fourier transform here is a rapidly
	decreasing function of $ \widebar{a } / \widebar{b} $, so  $\phi_{s, \varww}$ is smooth at these points. 
\end{proof}

\begin{lem}\label{lem: app: phi and Mphi}
	Let  $\mathrm{Re}(s) > 0$.  We have 
	\begin{align}\label{app: phiw = ...}
		\phi_{s, \varww} (1_2) = 0, \quad \phi_{s, \varww} (w) = \int_F \varww (x) \|x\|^s \nd x, 
	\end{align}
	and
	\begin{align}\label{app: M phi = ...}
		M \phi_{s, \varww} (1_2) = 0,  \quad M \phi_{s, \varww} (w) = \frac {\gamma (2s)} {\gamma (1-2s)} \int_F \varww (x) \|x\|^{-s} \nd x ,
	\end{align}
	where 
	\begin{align}\label{app: defn of gamma}
		\gamma (s) = \left\{ \begin{aligned}
			& \pi^{- s/2} \Gamma (s/2), & & \ \text{ if } F \text{ is real,}\\ 
			& 2 (2\pi)^{-s} \Gamma (s) , & & \ \text{ if } F \text{ is complex.}
		\end{aligned} \right.
	\end{align}
	
\end{lem}

	The formulae in \eqref{app: phiw = ...} follow  immediately from the definitions in \eqref{app: defn of phiw}. By \eqref{app: intertwining} and \eqref{app: defn of phiw},  
	\begin{align*}
		M  \phi_{s, \varww} (1_2) = \int_{ F } \int_{ F } \varww ( x  ) \|x \|^{s}   {\psi (  r x)} \nd \shskip x \, \nd r.
	\end{align*}
	By  the Fourier inversion formula, this is the value of $ \varww ( x  ) \|x \|^{s} $ at $x = 0$. However, this function is compactly supported in $F \smallsetminus \{0\}$, so $ M  \phi_{s, \varww} (1_2) = 0 $. As for $M \phi_{s, \varww} (w)$, it follows from \eqref{app: intertwining}  that
	\begin{align*}
		M \phi_{s, \varww} (w) = \int_{ F } \phi_{s, \varww} (w \shskip n(r) w )   \nd r,
	\end{align*}
	while for $r \neq 0$ we have
	\begin{align*}
		w \shskip n(r) w = \begin{pmatrix}
			1/r & -1 \\ 
			& r
		\end{pmatrix} w \begin{pmatrix}
			1  & -1/r \\ 
			& 1
		\end{pmatrix}
	\end{align*}
	so, on changing $r$ into $-1/r$, we obtain from \eqref{app: defn of phiw} that
	\begin{align*}
		M \phi_{s, \varww} (w) = \int_{ F  } \|r\|^{2s-1} \int_F \varww (x) \|x\|^{s}  {\psi ( r x  )} \nd x \,  \nd r .
	\end{align*}
	Since $\varww (x) \|x\|^{s}$ is smooth and compactly supported, its Fourier transform is of Schwartz class, and hence  the integral is convergent and analytic for all   $\mathrm{Re}(s) > 0$.  
	
Next, we formally change the order of integration. After this, the $r$-integral may be evaluated by Lemma \ref{lem: gamma, 0} below, with $ \varnu = s $ or $2s$,  then    the formula for $M \phi_{s, \varww} (w)  $ as in \eqref{app: M phi = ...} follows.  
	
	\begin{lem}\label{lem: gamma, 0} For $0 < \mathrm{Re} (  \varnu ) < \frac 1 2$  we have 
		\begin{align}\label{app: int = gamma, R}
			\int_{0}^{ \infty} x^{2 \varnu -1} (\ree (- x y) + \ree (xy ) ) \nd x =   \frac {\pi^{\frac 1 2-2 \varnu} \Gamma (\varnu)} { y^{2 \varnu} \Gamma \big(\frac 1 2 - \varnu \big)}   ,
		\end{align}
		and  
		\begin{align}\label{app: int = gamma, C}
			2 \int_{0}^{\infty} \int_{0}^{2 \pi } x^{2 \varnu -1}  \ree (- 2 x y \cos (\phi + \omega) ) \nd \phi \, \nd x  = \frac {(2\pi)^{1-2\varnu} \Gamma (\varnu)} {y^{2\varnu} \Gamma (1-\varnu)} ,
		\end{align}
		where $y \in (0, \infty)$ and $\omega \in [0, 2\pi)${\rm;}	the integrals are  convergent conditionally. 
	\end{lem}
	\begin{proof}[Proof of Lemma \ref{lem: gamma, 0}]
		By \cite[3.761 9]{G-R} and  \cite[Lemma 4.4]{Qi-BE}, the integrals in \eqref{app: int = gamma, R} and  \eqref{app: int = gamma, C},  respectively, are equal to $ 2 (2\pi y)^{-2\varnu} \Gamma (2\varnu) \cos (\pi \varnu) $ and $ 2 (2\pi y)^{-2 \varnu} \Gamma ( \varnu)^2 \sin (\pi \varnu) $, and we arrive at the right-hand sides of \eqref{app: int = gamma, R} and  \eqref{app: int = gamma, C} by the duplication and the reflection formulae for the gamma function. 
	\end{proof}
	
	However, the  change of the order of integration is not quite rigorous as the double integral does not converge absolutely. To justify this, we introduce an exponential factor $\exp (- 2 \pi \vepsilon |r|)$ or  $\exp (- 4 \pi \vepsilon |r|)$ in the $r$-integral to ensure absolute convergence. To evaluate the $r$-integral, we use the following Lemma \ref{lem: gamma, 1} instead of Lemma \ref{lem: gamma, 0}. Finally, we proceed to the limit as $\vepsilon   \ra 0$ to conclude the proof.  
	In view of  the duplication and the reflection formulae for the gamma function, \eqref{app: int = gamma, R, 1} and \eqref{app: int = gamma, C, 1} are the limiting forms of \eqref{app: int = gamma, R} and \eqref{app: int = gamma, C}, respectively. For the complex case, note that 
	\begin{align*}
		{_2F_1} \big( \varnu, \tfrac 1 2 - \varnu; 1;   1   \big) = \frac {\sqrt{\pi}} {\Gamma  ( 1 - \varnu  ) \Gamma \big(\frac 1 2 + \varnu \big)}, 
	\end{align*}
	by the Gauss formula. 
	
	\begin{lem}\label{lem: gamma, 1} Let $\vepsilon  > 0$.  For $  \mathrm{Re} (  \varnu ) > 0$, we have 
		\begin{align}\label{app: int = gamma, R, 1}
			\int_{0}^{ \infty} x^{2 \varnu -1} \exp (- 2 \pi \vepsilon x) (\ree (- x y) + \ree (xy ) ) \nd x =  \frac { 2 \Gamma (2\varnu) \cos (2 \varnu \arctan (y/\vepsilon))} { (2\pi)^{ 2 \varnu} (y^2 + \vepsilon^2)^{\varnu}  },    
		\end{align}
		and  
		\begin{align}\label{app: int = gamma, C, 1}
			\begin{aligned}
				2 \int_{0}^{\infty} \int_{0}^{2 \pi } x^{2 \varnu -1}   \exp (- 4 \pi \vepsilon x) & \ree (- 2 x y \cos (\phi + \omega) ) \nd \phi \, \nd x \\
				& \hskip -10pt = \frac {  \Gamma (2 \varnu)} { (4 \pi)^{ 2 \varnu-1} (y^2 + \vepsilon^2)^{\varnu} } {_2F_1} \bigg( \varnu, \tfrac 1 2 - \varnu; 1; \frac {y^2} {y^2 + \vepsilon^2 } \bigg),
			\end{aligned}
		\end{align}
		where $y \in (0, \infty)$ and $\omega \in [0, 2\pi)${\rm;}	the integrals are absolutely convergent. 
	\end{lem}
	\begin{proof}[Proof of Lemma \ref{lem: gamma, 1}]
		The formula \eqref{app: int = gamma, R, 1} is a direct consequence of \cite[3.944 6]{G-R}.  As for \eqref{app: int = gamma, C, 1}, we first compute the $\phi$-integral by Bessel's formula (see \cite[2.2 (1)]{Watson})
		\begin{align*}
			J_0 (x) = \frac 1 {2\pi} \int_0^{2\pi} \exp (i x \cos \phi) \nd \phi ,  
		\end{align*}
		so that the integral in \eqref{app: int = gamma, C, 1} turns into
		\begin{align*}
			4 \pi \int_{0}^{\infty}   x^{2 \varnu -1} \exp (- 4 \pi \vepsilon x) J_0 (4 \pi x y)   \nd x ,
		\end{align*}
		and this integral can be evaluated by \cite[13.2 (3)]{Watson}, giving the right-hand side of \eqref{app: int = gamma, C, 1}. 
	\end{proof} 

Finally, for the Whittaker function associated to $\phi_{s, \varww}$ we have the following lemma. 

\begin{lem}\label{lem: W = }
	Let $W_{s, \varww} = W_{\phi}$ with $\phi = \phi_{s, \varww}$. We have
	\begin{align}
		W_{s, \varww} (a(x)) = {\textstyle \sqrt{\|x\|}} \varww (x), \quad W_{s, \varww} (a(y) w) = {\textstyle \sqrt{\| y \|}} \int_{F    ^{\times}}  \varww (x)       B_{s} (xy)  \nd  x. 
	\end{align} 
\end{lem}

\begin{proof}
	By \eqref{app: Whittaker function}, \eqref{app: defn of phiw}, and Fourier inversion, we have
	\begin{align*}
		W_{s, \varww} (a (x)) & =   \int_{F} \phi_{s, \varww} (w \shskip n(r) a (x)   ) \overline{\psi ( r)} \nd r \\
		& = \|x\| \int_{F} \phi_{s, \varww} (\begin{pmatrix}
			1 & \\ & x
		\end{pmatrix}w \begin{pmatrix}
			1 & r \\ & 1
		\end{pmatrix}   ) \overline{\psi ( r x)} \nd r \\
		& = \| x \|^{\frac 1 2 - s} \int_{ F } \int_{ F } \varww ( v  ) \|v \|^{s}   {\psi (  r  v)} \nd \shskip v \, \overline{\psi ( r x)} \nd r \\
		& = {\textstyle \sqrt{\|x\|}} \varww (x). 
	\end{align*}
	The formula for $W_{s, \varww} (a(x) w)$ is precisely the kernel formula in Proposition \ref{lem: kernel formula}.
\end{proof}

\section{Proof in the Case $\mathrm{Re} (2 s) > 1$} Assume that $2 s$ is not an integer and that $\mathrm{Re} (2 s) > 1$. Let $\varww \in    \mathscr{C}^{\infty}_c (F^{\times}_{\infty})$, with $
\varww  = \prod_{v | \infty} \varww_{v}  $.  We choose $\phi_{\infty} $ to be the product     $ \prod_{v | \infty} \phi_{s, \varww_{v}}$.  In view of Lemma \ref{lem: app: phi and Mphi} and \ref{lem: W = }, if we change $\zeta$  into $-\zeta$, the sums in \eqref{app: LHS} and \eqref{app: RHS}, respectively, equal to
\begin{align}\label{app: LHS, 1}
	\frac {    1  } {   \RN (  \fra  )^{\frac 1 2} \RN (\frD)^{\frac 1 2+s }  \zeta_F(1+2s) }	\sum_{\gamma \shskip \in (\fra\frD)^{-1} \smallsetminus \{0\} }     \psi_f ( - \gamma \zeta )      \tau_s (\gamma \fra \frD)  \varww (\gamma ) ,
\end{align}
and
\begin{align}\label{app: RHS, 1}
	\begin{aligned}
		& \frac { \widetilde{\varww}_{s} (0) } {\RN  (  \frb)^{\frac 1 2+s}  }    + \frac {  \zeta_{F} (2s) \gamma_{F} (2s) } { \RN (  \frD )^{\frac 1 2} \zeta_{F} (1+2s ) \gamma_{F} (1-2s) }     \frac {  \widetilde{\varww}_{-s} (0) } { \RN  (\frb )^{\frac 1 2- s} }  \\
		  + \, &  \frac {    1  } {   \RN (  \frb  )^{\frac 1 2} \RN (\frD)^{\frac 1 2+s }  \zeta_F(1+2s) } \sum_{\gamma \shskip \in  (\frb \frD)^{-1} \smallsetminus \{0\} }    \psi_S  (   \gamma / \zeta ) \tau_{s} (\gamma   \frb \frD ) \widetilde{\varww}_{s} (\gamma)     ,  
	\end{aligned}
\end{align}
where $\gamma_{F} (s)$ is the product of the $\gamma_{v} (s)$ defined as in \eqref{app: defn of gamma}. Recall the functional equation for $\zeta_F $ (see \cite[\S XIV.8]{Lang-ANT}):
\begin{align*}
	\RN (\frD)^{s/2}	\zeta_{F} (s) \gamma_{F} (s) = \RN (\frD)^{(1 - s)/2}	\zeta_{F} (1-s) \gamma_{F} (1-s). 
\end{align*}
Hence 
\begin{align*}
	\frac {  \zeta_{F} (2s) \gamma_{F} (2s) } { \RN (  \frD )^{\frac 1 2} \zeta_{F} (1+2s ) \gamma_{F} (1-2s) } = \frac { \RN (\frD)^{-s} \zeta_F (1-2s) } {\RN (\frD)^s \zeta_F (1+2s)} . 
\end{align*}
Note that $\psi_f (- \gamma \zeta ) = \psi_{\infty} (   \gamma \zeta )$ for $\gamma, \zeta \in F$.   Since \eqref{app: LHS, 1} and \eqref{app: RHS, 1} are equal to each other,  we obtain  \eqref{app: Voronoi, tau s} after multiplying them by $  \RN (\frD)^{\frac 1 2+s} \zeta_F (1+2s) $. 

\section{Analytic Continuation}
To complete the proof, we need to verify the validity of \eqref{app: Voronoi, tau s} for all values of $s \in \BC$ by the principle of analytic continuation. To this end, it suffices to verify that both sides of  \eqref{app: Voronoi, tau s} are entire functions of $s$. 

Since $\varww  $ has compact support on $F^{\times}_{\infty}$, while $(\fra\frD)^{-1}$ is a lattice in $F_{\infty}$, the left-hand side is a finite sum and hence gives rise to an entire function of $s$. The function $\zeta_F (s)$ is analytic except for a simple pole at $s = 1$, hence the first sum on the right  is  entire, and at $s = 0$ it takes value 
\begin{align}\label{10eq: zeroth term s=0}
	 \sqrt{\RN(\frD)} \int_{F^{\sstimes}_{\scalebox{0.55}{$\infty$} } }  \varww (x) \big\{ \gamma^{(-1)}_F \log \big( \|x\|_{\infty} \RN (\frD) / \RN (\frb) \big) + 2 \gamma^{(0)}_F    \big\}  \nd x ,
\end{align}
for $\gamma^{(-1)}_F$ and $\gamma^{(0)}_F$  defined as in \eqref{2eq: zeta (s), s=1}.  Finally, the series on the right converges absolutely
and uniformly on compact subsets by Lemma \ref{lem: averages of tau s}  and \ref{lem: bounds for Hankel, local} below, with $V = 1$,  $c = \sigma + \vepsilon$, $d = 2$, and  $A = \sigma + 2$, so it converges to an entire function of $s$.  

\vskip 5pt

\subsection{Averages of Divisor Functions}
Actually, we can establish bounds, not just the convergence,  for certain averages of $\tau_{s} (\frn)$. See \cite[\S 4]{Qi-Wilton} and \cite[\S 4.3]{Qi-GL(3)} for their analogues in the cases of $\GL_2$ and $\GL_3$ cuspidal Fourier coefficients. 

\begin{lem}
	\label{lem: averages of tau s} 
Define   $N_{v} = 1$ if $F    _{v} = \BR$ and  $N_{v} = 2$ if $F    _{v} = \BC$. 	For   $V \in \BR_+^{|S_{\infty}|}$ and   $S \subset S_{\infty}$, define $\RN (V) = \prod_{ v | \infty } V_{v}^{N_{v}}$,
	$\|V \|_{S} = \prod_{ v \, \in S } V_{v}^{N_{v}}$, and 
	\begin{align} \label{4eq: defn of FS (T)}
		F_{\infty}^{S} (V) = \big\{ x \in F_{\infty} : \|x\|_{v} > V_{v}^{N_v} \text{if } v \in S, \|x\|_{v} \leqslant V_{v}^{N_v} \text{if } v \in S_{\infty} \smallsetminus S \big\} .
	\end{align}
Let $\sigma = |\mathrm{Re} (s)|$ and $0 \leqslant c - \sigma < 1 < d$.  Then for any $0 < \vepsilon     < d-1$ we have
	\begin{align}\label{3eq: average over ideals}  
		\mathop{\sum_{\sstyle  \gamma \shskip   \in   \Fx   \cap F_{\infty}^{S} (V)  }}_{  \sstyle  \gamma \fra   \subset \SO }   
		\frac { |\tau_s  (  \gamma \fra) | } {|\RN \gamma|^{c} \|\gamma \|_S^{d-c+\sigma}  } & = O_{\vepsilon, \shskip c, \shskip d, \shskip \sigma, \shskip F} \bigg(  \frac {   \RN ( \fra )^{1 + \sigma + \vepsilon} \RN(V)^{1-c + \sigma + \vepsilon} }  {\|V \|_S^{d-c + \sigma } }   \bigg),
	\end{align} 
	with the implied constant uniformly bounded for $\sigma$ in compact sets.
\end{lem}

\begin{proof}
	Firstly, by partial summation, we deduce from
	\begin{align*}
		\sum_{\RN (\frn) \leqslant X} 1 = O_F (X) 
	\end{align*}
	that
	\begin{align}\label{3eq: average of tau, 1}
		\sum_{\RN (\frn) \leqslant X} \frac {|\tau_s (\frn)|} {\RN (\frn)^c}  \leqslant \sum_{\RN (\frb) \leqslant X} \frac 1 {\RN (\frb)^{c + \sigma }} \sum_{\RN(\fra)  \leqslant X/ \RN (\frb)} \frac 1 {\RN (\fra)^{c-\sigma}} \Lt_F \frac {  X^{1-c+\sigma} \log X} {1-c+\sigma}, 
	\end{align}
	for $X \geqslant 2$, provided that $0 \leqslant c - \sigma < 1$. Next, we use \eqref{3eq: average of tau, 1} as a substitute of (4.3)  in \cite{Qi-Wilton} and apply his Lemma 4.1 to prove for any  $V \in \BR_+^{|S_{\infty}|}$  (see also the proof of \cite[Lemma 4.10]{Qi-GL(3)})
	\begin{align} \label{3eq: average of tau, 2}
		\mathop{\sum_{\sstyle  \gamma \shskip   \in   \Fx   \cap F_{\infty}^{\text{\O}} (V)  }}_{  \sstyle  \gamma \fra   \subset \SO }   
		\frac { |\tau_s  (  \gamma \fra) | } {|\RN \gamma|^{c}  } & = O_{\vepsilon, \shskip c,  \shskip \sigma, \shskip F} \big(       \RN ( \fra )^{1 + \sigma + \vepsilon} \RN(V)^{1-c - \sigma + \vepsilon}     \big), 
	\end{align} 
	which is an analogue of his Lemma 4.2. Finally, we proceed as in the proof of Lemma 4.3 in \cite{Qi-Wilton} to derive \eqref{3eq: average over ideals}  from  \eqref{3eq: average of tau, 2}. It is easy to verify the uniformity in $\sigma$ at each step. 
\end{proof}

\subsection{Estimates for the Hankel Transform}

Finally, we have  the following crude but uniform estimates for the Hankel transform. For brevity, we will suppress the place $v$ from our notation.

\begin{lem}\label{lem: bounds for Hankel, local}
Let $\sigma = |\mathrm{Re} (s)|$.	For $\varww (x) \in C_c^{\infty} (\Fx )$ we have 
	\begin{align*}
		\int_{\Fx } \varww  (x) B_s (x y) d x \Lt_{s, \shskip  \vepsilon , \shskip A , \shskip \varww}   \left\{ \begin{aligned}
			& 1/ {\|y\|^{   \sigma + \vepsilon }}, & & \text{ if } \|y\| \leqslant 1,  \\
			& 1 / {\|y\|^{A}} , & & \text{ if } \|y\| > 1,  
		\end{aligned}      \right.  
	\end{align*}
	for any  $\vepsilon > 0$ and $A \geqslant 0$, with  the implied constants uniformly bounded for $s$ in compact sets. 
\end{lem}

\begin{proof} 
	For fixed $s$ the estimates follow immediately from Theorem 3.12, 3.15 and Proposition 3.14, 3.17 in \cite{Qi-Bessel}. However, to prove the uniformity in $s$, we require uniform bounds and asymptotics for the Bessel kernel $B_s $.

	Now fix $c \geqslant 1$ and let  $|s| \leqslant c$.   
	
	Proceeding as in  \cite[\S 5.1]{Qi-Wilton},  by estimating the Mellin--Barnes type integrals of certain gamma factors, for $|x| \leqslant c^4 $ we deduce the bounds
	\begin{align*}
		B_s (x) \Lt_{c, \shskip \vepsilon} 1 / \|x\|^{\sigma +\vepsilon}.
	\end{align*}
	It is critical that the integral contours therein can be chosen fixed for given $c$ and $\vepsilon$. Then  the first uniform estimate follows directly.  
	
	Next, we invoke the formulae 
	\begin{align*}
		& B_{s} (x)  = \pi i \big( e^{\pi i s}  H^{(1)}_{2 s} (4\pi \sqrt{x}) -  e^{- \pi i s}   H^{(2)}_{2 s} (4\pi \sqrt{x}) \big), \\
		& B_{s} (-x )  =    {4 \cos (\pi s  )}    K_{2 s } (4 \pi \sqrt {x }) ,
	\end{align*}
	for $x \in \BR_+$, and
	\begin{align*}
		B_{s} (z) = \pi^2 i \big( e^{2 \pi i s} H^{(1 )}_{ 2s   } (4 \pi  \sqrt{z} ) H^{(1 )}_{  2 s    }   (4 \pi  {\textstyle \sqrt{\widebar{z}}} )  - e^{- 2 \pi i s} H^{(2 )}_{ 2s   } (4 \pi  \sqrt{z} ) H^{(2 )}_{  2 s    }   (4 \pi  {\textstyle \sqrt{\widebar{z}}} )    \big), 
	\end{align*} 
	for $z \in \BC^{\times}$; see  \cite[(3.61 (1), (2))]{Watson}. 
	By \cite[\S 7.13.1, Ex. 13.2]{Olver}, we deduce the uniform asymptotic formulae:
	\begin{align*}
		& B_s (x) =   \sum_{ \pm} \frac {\ree (\pm (2 \sqrt{x} + 1/8))} {x^{1/4}} \sum_{k =0}^{K-1} \frac {(\pm)^k A_k (s)} {x^{k/2}} + O_{c, \shskip K} \bigg(  \frac 1 {x^{(2K+1)/4}} \bigg),  \\
		& B_s (-x) = O_{c} \bigg( \frac {\exp (-4\pi \sqrt{x})} {x^{1/4}} \bigg), 
	\end{align*}
	for $x > c^4$, and 
	\begin{align*}
		& B_s (z) =   \sum_{ \pm} \frac {\ree (\pm 2 \shskip \Tr \sqrt{z})} {|z|^{  1 / 2}} \mathop{\sum\sum}_{ k, \shskip l = 0}^{  K-1}  \frac {(\pm)^{k+l} A_k (s) A_l (s) } {z^{k/2} \widebar{z}^{l/2} } + O_{c, \shskip K} \bigg(  \frac 1 {|z|^{(K+1)/2}} \bigg),   
	\end{align*}
	for $|z| > c^4$, 
	where $K$ is any non-negative integer, and the coefficient $A_k (s)$ is a certain polynomial in $  s$ of degree $2 k$. Then the second uniform estimate follows from repeated partial integration (we obtain Fourier integrals on letting $\sqrt{x}$ or $\sqrt{z}$ be the new variable) or directly from the exponential decay  (in the real case). 
\end{proof}


\end{document}